\documentclass[10pt]{article}

\usepackage{amsmath,amssymb,amscd,amsthm,makeidx,txfonts,graphicx,url,bm}
\usepackage{a4wide,pstricks,xypic}
\usepackage[backref]{hyperref}

\numberwithin{equation}{subsection}

\let\oldmarginpar\marginpar
\renewcommand\marginpar[1]{\-\oldmarginpar[\raggedleft\footnotesize #1]
{\raggedright\footnotesize #1}}

\makeindex

\xyoption{all}
\input{xypic}

\newtheorem{theorem}{Theorem}[section]
\newtheorem{proposition}[theorem]{Proposition}
\newtheorem{corollary}[theorem]{Corollary}

\newtheorem{lemma}[theorem]{Lemma}

\theoremstyle{remark}
\newtheorem{remark}[theorem]{Remark}

\theoremstyle{definition}

\newcounter{margin}
{\end{itshape}  \bigskip}

\def\beq{\begin{eqnarray}}
\def\eeq{\end{eqnarray}}
\def\bes{\begin{eqnarray*}}
\def\ees{\end{eqnarray*}}

\DeclareMathOperator{\Aut}{Aut} 
 \DeclareMathOperator{\Hom}{Hom}
\DeclareMathOperator{\End}{End}
\DeclareMathOperator{\id}{id}

\DeclareMathOperator{\Stab}{Stab}

\def\C{\mathbb{C}}

\def\F{\mathbb{F}}
\def\Q{\mathbb{Q}}

\def\Z{\mathbb{Z}}

\newcommand{\nc}{\newcommand}

\nc{\op}[1]{\mathop{\mathchoice{\mbox{\rm #1}}{\mbox{\rm #1}}
{\mbox{\rm \scriptsize #1}}{\mbox{\rm \tiny #1}}}\nolimits}
\nc{\al}{\alpha}

\nc{\ep}{\varepsilon} \nc{\ga}{\gamma} \nc{\Ga}{\Gamma}
\nc{\la}{\lambda} \nc{\La}{\Lambda} \nc{\si}{\sigma}
\nc{\Sig}{{\Gamma}} \nc{\Om}{\Omega} \nc{\om}{\omega}

\nc{\SL}{{\rm SL}} \nc{\GL}{{\rm GL}} \nc{\PGL}{{\rm PGL}}
\nc{\PSL}{{\rm PSL}}
\nc{\G}{{\rm G}}

\nc{\cpt}{{\op{cpt}}} \nc{\Dol}{{\op{Dol}}} \nc{\DR}{{\op{DR}}}
\nc{\B}{{\op{B}}} \nc{\Triv}{\op{Triv}} \nc{\Hod}{{\op{Hod}}}
\nc{\Log}{{\op{Log}}} \nc{\Exp}{{\op{Exp}}} \nc{\Est}{E_{\op{st}}}
\nc{\Hst}{H_{\op{st}}} \nc{\Left}[1]{\hbox{$\left#1\vbox to
  10.5pt{}\right.\nulldelimiterspace=0pt \mathsurround=0pt$}}
\nc{\Right}[1]{\hbox{$\left.\vbox to
  10.5pt{}\right#1\nulldelimiterspace=0pt \mathsurround=0pt$}}
\nc{\LEFT}[1]{\hbox{$\left#1\vbox to
  15.5pt{}\right.\nulldelimiterspace=0pt \mathsurround=0pt$}}
\nc{\RIGHT}[1]{\hbox{$\left.\vbox to
  15.5pt{}\right#1\nulldelimiterspace=0pt \mathsurround=0pt$}}

\nc{\bee}{{\bf E}} \nc{\bphi}{{\bf \Phi}}

\begin{document}

\title{On the divisibility of $\#\Hom(\Gamma,G)$ by $|G|$}

\author{ Cameron Gordon\\
{\it University of Texas at Austin} \\ {\tt
  gordon@math.utexas.edu}\\ \\
Fernando Rodriguez-Villegas
\\ 
{\it University of Texas at Austin} \\ {\tt
  villegas@math.utexas.edu}\\ \\
 }

\pagestyle{myheadings}

\maketitle

\begin{abstract} 
  We extend and reformulate a result of Solomon on the divisibility of
  the title. We show, for example, that if $\Gamma$ is a finitely
  generated group, then $|G|$ divides $\#\Hom(\Gamma,G)$ for every
  finite group $G$ if and only if $\Gamma$ has infinite
  abelianization. As a consequence we obtain some arithmetic
  properties of the number of subgroups of a given index in such a
  group $\Gamma$.
\end{abstract}

\section{Introduction}
Let $G$ be a finite group.  For any non-negative integer $g$ and a
fixed $z\in G$ consider the set
\begin{equation}
  \label{def-U}
  U_z:=\{(x_1,y_1,\ldots,x_g,y_g)\in G^{2g} \;|\; [x_1,y_1]\cdots[x_g,y_g]=z\},
\end{equation}
where $[x,y]:=xyx^{-1}y^{-1}$.  We have~\cite{Se1}
\begin{equation}
  \label{U-card}
\#U_z=\sum_\chi \left(\frac{|G|}{\chi(1)}\right)^{2g-1}\chi(z),
\end{equation}
where the sum is over all irreducible characters of $G$ (formulas of
this sort were already known to Frobenius). 

In particular, for $z=1$ we obtain
\begin{equation}
\label{frobenius-1}
\#U_1=|G|\sum_\chi \left(\frac{|G|}{\chi(1)}\right)^{2g-2}.
\end{equation}
Since $\chi(1)$ divides $|G|$ for every $\chi$, it follows that $|G|$
divides $\#U_1$ when $g>0$. The purpose of this note is to give a more
direct proof of a generalization of this observation.

We may interpret $U_1$ as the set $\Hom(\pi_1(\Sigma_g),G)$ of
homomorphisms of the fundamental group of a Riemann surface $\Sigma_g$
of genus $g$ to $G$. The question is then the following. Given a
finitely generated group $\Gamma$, say, when do we have that $|G|$
divides $\Hom(\Gamma,G)$ for every finite group $G$?  It turns out
that the answer is quite simple: precisely when $\Gamma$ has an
infinite abelianization (see Corollary~\ref{infinite}).

As it happens, the answer to our question is buried in a paper of
Solomon's~\cite{Sol}, which was the main motivation for the
present note. Solomon's proof assumes that $\Gamma$ has positive
deficiency (i.e., has a presentation with strictly more generators
than relations), but it is easy to see that it can be made to apply
more generally to any group with infinite abelianization. However, his
proof is lengthy and somewhat hard to read; the proof we give is short
and, we hope, more perspicuous. Solomon deduces his result from a more
general statement about homomorphisms from a free group to $G$ that
send specified elements into specified conjugacy classes. As an
immediate corollary to our theorem we get a generalization of this in
which the free group is replaced by any group whose abelianization is
infinite.

\section{$\cal G$-torsors}
Let $G$ be a finite group. A $G$-set $X$ is a set $X$ on which $G$
acts.  We denote by $X/G$ the set of orbits under the action. Let
$\cal G$ be the bundle of groups $\Stab_G(x)$ on $X$. A {\it $\cal
  G$-torsor} of $X$ is a surjective map of sets $\pi: Y \rightarrow
X$, for some set $Y$, such that each fiber $Y_x$ is a principal
homogeneous space for $\Stab_G(x)$. In other words, there is an action
of $\Stab_G(x)$ on $Y_x$ for every $x\in X$ such that for any fixed
$y\in Y_x$ the map
$$
\begin{array}{ccc}
 \Stab_G(x) & \rightarrow & Y_x\\
 s & \mapsto & sy
\end{array}
$$
is a bijection.  It is worth stressing that we {\it do not} require a
global action of $G$ on $Y$ inducing the action of $\Stab_G(x)$ on
$Y_x$. 

We will write $\#$ for the cardinality of a set and reserve $|\cdot|$
for the order of a group.
\begin{lemma}
\label{mass}
Let $X$ be a $G$-set and $\pi:Y \rightarrow X$ a $\cal
G$-torsor. Then

(i) $X$ is finite if and only if $Y$ is finite.

(ii) In this case, $\#Y/|G|=\#(X/G)$. In particular $\#Y/|G|$ is an
integer or, equivalently, $|G|$ divides $\#Y$.
\end{lemma}
\begin{proof}
  The first assertion (i) is clear since $G$ is finite. As for (ii)
  note that
$$
\#Y=\sum_{x\in X}\#Y_x=\sum_{x\in X}|\Stab_G(x)|.
$$
Since $|\Stab_G(x)|$ is constant on orbits in $X/G$ we can group terms in
the sum by orbits. The contribution to the sum of an orbit $[ x]$
is then $\#[x]|\Stab_G(x)|=|G|$. Hence 
$$
\#Y=\sum_{[x]}|G|=|G|\cdot\#(X/G).
$$
and the lemma is proved.
\end{proof}

\section{Homomorphisms}
Let $\Gamma$ be a group. The set $X:=\Hom(\Gamma,G)$, where $G$ is a
finite group, is then finite if $\Gamma$ is a finitely generated. The
group $G$ acts on $X$ by conjugation. The quotient
$\#\Hom(\Gamma,G)/|G|$ that we are interested in can be thought of as
a weighted count of homomorphisms from $\Gamma$ to $G$. Namely,
$$
\frac1{|G|}\#\Hom(\Gamma,G)=
\sum_{[\phi]}\frac 1{|\Stab_G(\phi)|},
$$
where $[\phi]$ runs through the $G$-orbits of $\Hom(\Gamma,G)$.

We also have an action of $\Aut(\Gamma)$ on $X$ by
$$
\phi^\sigma(\gamma):=\phi(\sigma^{-1} \gamma), \qquad \sigma \in
\Aut(\Gamma), \quad \phi \in X, \quad \gamma \in \Gamma.
$$

Fix $\sigma\in \Aut(\Gamma)$ and let $X_\sigma:=\Hom_\sigma(\Gamma,G)
\subseteq \Hom(\Gamma,G)$ be the subset of $\phi \in X$ which are
fixed by $\sigma$ up to conjugation by $G$. The two actions, of $G$
and $\Aut(\Gamma)$, on $X$ clearly commute. In particular, the action
of $\Aut(\Gamma)$ passes to the quotient $X/G$ and $X_\sigma$ is a
$G$-set.

Let $\Gamma \rtimes_\sigma \Z$ be the semidirect product of $\Gamma$
and $\Z$ induced by the map $\Z\rightarrow \Aut(\Gamma)$ that takes
$1$ to $\sigma$.  A homomorphism $\Phi \in \Hom(\Gamma \rtimes_\sigma
\Z,G)$ is uniquely determined by the pair $(\phi,g)$, where
$\phi:=\left .\Phi \right|_\Gamma$ and $g\in G$ is the image of
$(1,1)\in \Gamma \rtimes_\sigma \Z$.  The necessary and sufficient
condition for $(\phi,g)$ to arise from a $\Phi$ in this way is
$$
\phi^\sigma(\gamma)=g\phi(\gamma)g^{-1}, \qquad \gamma \in \Gamma.
$$
In particular, $\phi\in \Hom_\sigma(\Gamma,G)$.  It follows that if we
fix one pair $(\phi,g)$ then the set of all other pairs $(\phi,g')$
are given by setting $g'=gs$ with $s\in\Stab_G(\phi)$. We may hence
define an action of $\Stab_G(\phi)$ on these pairs by setting $s\cdot
(\phi,g):=(\phi,gs^{-1})$ This proves the following.

\begin{proposition}
\label{semidirect}
With the above notation for any finite group $G$ the map
$$
\Hom(\Gamma \rtimes_\sigma \Z,G) \rightarrow \Hom_\sigma(\Gamma,G)
$$
given by restriction is a $\cal G$-torsor of $\Hom_\sigma(\Gamma,G)$.
\end{proposition}
Combining Proposition~\ref{semidirect} with Lemma~\ref{mass} we obtain
the following corollary
\begin{corollary}
\label{divisibility}
If $\Gamma \rtimes_\sigma \Z$ is finitely generated then

(i) $\Hom_\sigma(\Gamma,G)$ is finite,

(ii) moreover,
$$
\frac1{|G|}\#\Hom(\Gamma \rtimes_\sigma \Z,G) =
\#(\Hom_\sigma(\Gamma,G)/G).
$$
\end{corollary}

We should point out that for a finitely generated group $\tilde
\Gamma$ to be isomorphic to a $\Gamma \rtimes_\sigma \Z$ for some
$\Gamma$ and $\sigma$ is equivalent to having infinite
abelianization. If $\tilde \Gamma$ has finite abelianization $A$ then
picking, say, $G=\Z/pZ$ with $p$ a prime not dividing $|A|$ we see
that $|G|$ does not always divide $\#\Hom(\tilde \Gamma,G)$.  We have
proved the following.
\begin{corollary}
\label{infinite}
A finitely generated group $\tilde \Gamma$ has the property that $|G|$
divides $\#\Hom(\tilde \Gamma,G)$ for every finite group $G$ if and
only if it has infinite abelianization.
\end{corollary}

We can extend the previous results as follows.  Let ${C_i}$ be an
indexed collection of (not necessarily distinct) conjugacy classes in
$G$, and let ${S_i}$ be a collection of subsets of $\Gamma$. Let
$\Hom'(\Gamma \rtimes_\sigma \Z,G)\subseteq \Hom\Gamma \rtimes_\sigma
\Z,G)$ consist of those homomorphisms $\Phi$ such that $\Phi(S_i)
\subseteq C_i$ for all $i$.  Define $\Hom'_\sigma(\Gamma,G)\subseteq
\Hom_\sigma(\Gamma,G)$ similarly and let $\pi'$ be the restriction of
$\pi$ to $\Hom'(\Gamma \rtimes_\sigma \Z,G)$.

\begin{theorem}
\label{main}
The map $\pi'$ takes $\Hom'(\Gamma \rtimes_\sigma \Z,G)$ to
$\Hom'_\sigma(\Gamma,G)$ and
$$
\pi': \Hom'(\Gamma \rtimes_\sigma \Z,G)\to \Hom'_\sigma(\Gamma,G)
$$
is a $\cal G$-torsor of $ \Hom'_\sigma(\Gamma,G)$.
\end{theorem}
\begin{proof}
  Note that $\Hom'_\sigma(\Gamma,G)$ is indeed a sub-$G$-set of
  $\Hom_\sigma(\Gamma,G)$ since the $C_i$ are conjugacy classes and
  $\Hom'(\Gamma \rtimes_\sigma \Z,G)=
  \pi^{-1}(\Hom'_\sigma(\Gamma,G))$. Now the claim follows from
  Proposition~\ref{semidirect}.
\end{proof}
As before we obtain
\begin{corollary}
\label{main-corollary}
Let $\Gamma \rtimes_\sigma \Z$ be finitely generated. Then
$$
\frac1{|G|}\#\Hom'(\Gamma \rtimes_\sigma \Z,G)=
\#(\Hom'_\sigma(\Gamma,G)/G).
$$
\end{corollary}
\begin{remark}
  The fact that $|G|$ divides $\#\Hom'(\Gamma \rtimes_\sigma \Z,G)$
  when $\Gamma \rtimes_\sigma \Z$ is a free group of rank $n$, the
  conjugacy classes $C_i$ are indexed by ${1,2,...,m}$ with $m < n$,
  and each $S_i$ is a singleton is the main result of~\cite{Sol}.
\end{remark}

\section{Examples}
We illustrate some of the issues concerning the ratio $\#\Hom'(\tilde
\Gamma,G)/|G|$ with a few examples.  

(1) Take $\tilde \Gamma$ to be the free group $F$ in $k-1$ generators
given as $\langle x_1,\ldots,x_k \,|\, x_1\cdots x_k =1\rangle$ for
some $k>1$ and $S_i=\{x_i\}$ for $i=1,2,\ldots,k$. Fix a finite group
$G$ and conjugacy classes $C_1,\ldots,C_k$ of $G$ and let $\Hom'(F,G)$
be as above. In general there is no reason to expect $\#\Hom'(F,G)$ to
be divisible by $|G|$, though the denominator of the quotient is often
much smaller than $|G|$.

By a formula extending~\eqref{U-card} we have
\begin{equation}
\label{frobenius}
\#\Hom'(F,G)=\sum_\chi\frac{\chi(1)^2}{|G|}\,  \prod_{i=1}^k
f_\chi(C_i),
\end{equation}
where the sum is over all irreducible characters of $G$ and for a
conjugacy class $C$ we define 
$$
f_\chi(C):=\frac{\#C\, \chi(x)}{\chi(1)}, \qquad \qquad x\in C.
$$
It is known that $f_\chi(C)$ is an algebraic integer.

For example, take $G=S_n$, the symmetric group in $n$ letters, and let
$C_i$, for $i=1,2,\ldots,k$, be the conjugacy class of an $n$-cycle
$\rho:=(12\cdots n)$. The irreducible characters $\chi_\lambda$ of
$S_n$ are parametrized by partitions $\lambda$ of $n$ and
$\chi_\lambda(\rho)=(-1)^r$ if $\lambda=(n-r,1,\ldots,1)$ is the
$r$-th hook and $\chi_\lambda(\rho)=0$ otherwise.

A short calculation then shows that 
$$
\frac
1{|G|}\#\Hom'(F,G)
=\frac1{n^2}\sum_{r=0}^{n-1}((-1)^rr!(n-r-1)!)^{k-2},
$$
which is typically not an integer; for example, for $k=2$ it equals
$n^{-1}$. In fact, for $k>2$ this quantity is an integer unless $n$ is
a prime not dividing $k-1$, in which case the denominator equals
$n$. (We thank J. Gunther for showing us a proof of this fact.)

(2) In the previous example take $k=4$ and $G=\SL_2(\F_q)$ for some
$q=p^r$ with $p>2$ a prime.  For simplicity pick $C_i$, for
$i=1,\ldots,4$, to be the conjugacy class of a diagonal matrix with
eigenvalues $\lambda_i,\lambda_i^{-1} \in \F_q^\times\setminus \{\pm
1\}$. Then using the known description of the irreducible characters
of $G$ we find after some calculation that in this case
$$
\frac
1{|G|}\#\Hom'(F,G)=q^2+4q+1+a\frac{q^2}{q-1},
$$
where
$$
a:=\tfrac12\#\{(\varepsilon_1,\ldots,\varepsilon_4)\in (\pm1)^4 \,| \,
\lambda_1^{\varepsilon_1}\cdots\lambda_4^{\varepsilon_4}=1\}.
$$
If $q$ is sufficiently large there will be a choice of eigenvalues
$\lambda_1,\ldots,\lambda_4$ such that $a=0$, in which case
$\#\Hom'(F,G)/|G|=q^2+4q+1$ is polynomial in $q$. In fact, with this
choice of {\it generic} eigenvalues the action of $G$ by conjugation
on $\#\Hom'(F,G)$ is actually free explaining the divisibility.  (See
\cite{HLRV} for more details).

(3) Consider the group $\tilde \Gamma:=\langle
u_1,\ldots,u_n,z_1,\ldots,z_k \,|\, w(u_1,\ldots,u_n)\cdot z_1\cdots
z_k\rangle$, where $n\geq 1$ and $w$ is an arbitrary word in
$u_1,\ldots,u_n$ that belongs to the commutator of the free group
$\langle u_1,\ldots,u_n\rangle$. Pick a non-trivial homomorphism
$\psi: \tilde \Gamma\to \Z$ with $\psi(z_i)=0$ for $i=1,2,\ldots,k$
and let $\Gamma:=\ker \psi$.  Take $S_i=\{z_i\}$ for $i=1,2,\ldots,k$
and pick arbitrary conjugacy classes $C_1,\ldots,C_k$ in $G$. Then by
Corollary~\ref{main-corollary} $|G|$ divides $\#\Hom'(\tilde
\Gamma,G)$.  Concretely, the number of solutions to the equation
$$
w(u_1,\ldots,u_n)\cdot z_1\cdots z_k = 1, \qquad u_i\in G, \qquad
z_i\in C_i
$$
is divisible by $|G|$ for any finite group $G$. If we take $k=0$ and
$w=[x_1,y_1]\cdots,[x_g,y_g]$ this yields the case of a Riemann
surface~\eqref{frobenius-1} considered in the introduction.

(4) Consider the case where $\tilde \Gamma=\langle x,y \,| \,
w(x,y)\rangle$ is a two generator, one-relator group. Clearly, $\tilde
\Gamma$ has infinite abelianization and hence $\#\Hom(\tilde
\Gamma,G)/|G| \in \Z$ by Corollary~\ref{infinite}. We can express
this quantity in terms of irreducible characters of $G$ as follows.
$$
\frac 1{|G|}\#\Hom(\tilde\Gamma,G)=\sum_\chi s_\chi(w)\, \chi(1),
$$
where for an irreducible character $\chi$ of $G$
$$
s_\chi(w):=\frac1{|G|^2}\sum_{x,y\in G}\chi(w(x,y))
$$
(see~\cite{Se1}). For example, if $w(x,y)=xyx^{-1}y^{-1}$ then
$s_\chi(w)=\chi(1)^{-1}$. We may wonder if it is always the case that
$s_\chi(w)\chi(1) \in \Z$. This turns out not to be true.  In fact,
the numbers $s_\chi(w)\chi(1)$ appear to be fairly arbitrary algebraic
numbers in general.

Consider for example, $w=x^2y^2x^{-2}y^{-2}$ and
$G=\PSL_2(\F_{11})$. Then the values of $s_\chi(w)\chi(1)$ for the
different irreducible characters (computed using Magma) and their sum
are as follows
$$
1+29/6+29/6+551/33+547/33+296/15+w+w'=112,
$$
where $w,w'\in \Q(\sqrt 5)$ are the roots of 
$$
3025x^2 - 146190x +1766236=0.
$$ 
(The corresponding character dimensions $\chi(1)$ are
$1,5,5,10,10,11,12,12$.)

While testing this question numerically we discovered that for some
words, for example $w=x^2yxy^2$, it did indeed seem that
$s_\chi(w)\chi(1)\in \Z$ for all $\chi$. It turns out not difficult to
show why. Consider more generally $w=x^{-m}yx^ny^{-1}$ (take
$m=-2,n=1$ and replace $x$ by $xy$ to obtain a conjugate of the word
$x^2yxy^2$ just mentioned). The resulting group $\tilde \Gamma$ is
called a {\it Baumslag--Solitar} group.

Let $V$ be an irreducible representation of $G$ over $\C$ with
character $\chi$. For a fixed $u\in G$ let
$$
U:=\frac1{|G|}\sum_{y\in G} yuy^{-1}\in \End(V).
$$
It is easy to check that $U$ preserves the $G$ action on $V$. Hence by
Schur's lemma it must be of the form $\gamma \id_V$. Computing traces
we find that $\gamma=\chi(u)/\chi(1)$. Setting $u=x^n$ and computing
the trace of $x^{-m}U$ we finally find that
$$
\chi(1)s_\chi(w)=\langle \Psi^m\chi,\Psi^n\chi\rangle,
$$
where $\langle \alpha,\beta\rangle:=1/|G|\sum_{x\in G}\alpha(x)\bar
\beta(x)$ is the standard inner product of class functions and
$\Psi^m(\chi)(x):=\chi(x^m)$ is the $m$-th Adams operations on
(virtual) characters.  In particular, $\chi(1)s_\chi(w)$ is an integer
for all $\chi$.

More directly, if $n=1$ we can count the solutions to $yxy^{-1}=x^m$
with $x,y\in G$ as follows. Group the solutions according to the
conjugacy class $C$ of $x$. This class must satisfy $C^m=C$ and
contributes precisely $|G|$ to the total since for fixed $x\in C$ we
have that $y$ lies in a coset of the centralizer of $x$. Hence in this
case
$$
\frac 1{|G|}\#\Hom(\tilde\Gamma,G)=\#\{C^m=C\}.
$$
For example, for $G=\SL_2(\F_p)$  it is not hard to verify 
that when $m$ is even this number equals
$$
1+\delta_p(m)+
\tfrac12\sum_{\varepsilon_1, \varepsilon_2=\pm1} 
(\gcd(p+\varepsilon_1, m+\varepsilon_2)-1),
$$
where $\delta_p(m)=2$ if $m$ is a square modulo $p$ and is zero
otherwise. 

\section{Subgroups}

It is a consequence of the exponential formula in combinatorics that
for a finitely generated group $\tilde \Gamma$ we have
$$
F(x):=\sum_{n\geq 0}\#\Hom(\tilde \Gamma,S_n)\,\frac{x^n}{n!}
=\exp\left(\sum_{n\geq 1}u_n(\tilde \Gamma)\,\frac{x^n}n\right),
$$
where $u_n(\tilde \Gamma)$ denotes the number of subgroups of $\tilde
\Gamma$ of index $n$ (see for example~\cite{Me},\cite{Stan}). By
Corollary~\ref{infinite} if $\tilde \Gamma$ has infinite
abelianization then the series on the left hand side has integer
coefficients and hence may be written as an infinite product
$$
F(x)=\prod_{n\geq 1}(1-x^n)^{-v_n(\tilde \Gamma)},
$$
for certain integers $v_n(\tilde \Gamma)$. Comparing the two
expressions we find that
\begin{equation}
  \label{u-v}
  u_n=\sum_{d|n}d\,v_d
\end{equation}
and by M\"obius inversion
\begin{equation}
  \label{v-u}
  v_n=\frac1n \sum_{d|n}\mu\left(\frac nd\right) u_d.
\end{equation}

As before we can write $\tilde \Gamma=\ \Gamma \rtimes_\sigma \Z$ and
by Corollary~\ref{divisibility} we get
$$
F(x)=\sum_{n\geq 0}\#(\Hom_\sigma(\Gamma,S_n)/S_n)\,x^n.
$$
We can interpret the $n$-th coefficient of this series as the number
of isomorphism classes of $\sigma$-equivariant actions of $\Gamma$ on
a set of $n$ elements. It follows that the exponents $v_n(\tilde
\Gamma)$ count the number of such actions which are indecomposable
(i.e., transitive). In particular, $v_n(\tilde \Gamma)$ is a
non-negative integer.

If $\sigma$ is trivial, equivalently if $\tilde \Gamma =\Gamma \times
\Z$ is a direct product, then isomorphism classes of transitive
actions of $\Gamma$ on $n$ objects corresponds to conjugacy classes of
subgroups of $\Gamma$ of index $n$. (This interpretation of
$v_n(\tilde \Gamma)$ appears as exercise 5.13 (c) in~\cite{Stan}
with a different suggested proof.)

From~\eqref{u-v}, we obtain the following.
\begin{proposition}
Let $\tilde \Gamma$ be a finitely generated group with infinite
abelianization and let $u_n:=u_n(\tilde \Gamma)$ be its number of subgroups
of index $n$. Then for every prime number $p$ we have
$$
u_{p^{k+1}}\equiv u_{p^k}\bmod p^{k+1}.
$$
\end{proposition}
Here is a  short table of the numbers $u_n$ in the case of
$\tilde \Gamma = \pi_1(\Sigma_g)$, the fundamental group of a genus
$g$ Riemann surface,

\bigskip
\begin{center} \begin{tabular}{|r|rrrrr|} \hline
\bf $g\backslash n$ & 1 & 2 & 3 & 4 & 5 \\
\hline
1 & 1 &  3 &  4 &  7 &  6 \\
2 & 1 &  15 &  220 &  5275 &  151086\\
3 & 1 &  63 &  7924 &  2757307 &  2081946006\\
4 & 1 &  255 &  281740 &  1542456475 &  29867372813886\\
5 &  1 &  1023 &  10095844 &  882442672507 &  429988374084026406\\
\hline
\end{tabular}
\end{center}
\bigskip
and the corresponding numbers $v_n$

\bigskip
\begin{center} \begin{tabular}{|r|rrrrr|} \hline
\bf $g\backslash n$ & 1 & 2 & 3 & 4 & 5 \\
\hline
1 & 1 &  1 &  1 &  1 &  1\\
2& 1 &  7 &  73 &  1315 &  30217\\
3& 1 &  31 &  2641 &  689311 &  416389201\\
4 & 1 &  127 &  93913 &  385614055 &  5973474562777\\
5& 1 &  511 &  3365281 &  220610667871 &  85997674816805281\\
\hline
\end{tabular}
\end{center}

\end{document}